\documentclass[11pt,reqno]{amsart}

\usepackage{amsthm,amsmath,amssymb}
\usepackage{graphicx}
\usepackage{float}
\usepackage[colorlinks=true,
linkcolor=blue,
urlcolor=blue,
citecolor=blue]{hyperref}
\usepackage{mathtools}
\usepackage{relsize}
\usepackage{tikz}
\usepackage{enumerate}
\usepackage[toc,page]{appendix}
\usepackage{pdflscape}
\usepackage{multirow}
\usepackage{adjustbox,expl3,etoolbox} 
\usepackage[margin = 3cm]{geometry}
\usepackage{rotating}
\usepackage{tablefootnote}
\usepackage{longtable}
\usepackage{comment}
\usepackage{xcolor}
\usepackage{mathrsfs}
\usepackage{bbm}
\usepackage[all]{xy}

\newtheorem{thm}{Theorem}[section]
\newtheorem{lem}[thm]{Lemma}
\theoremstyle{definition}
\newtheorem{defn}[thm]{Definition}
\newtheorem{prop}[thm]{Proposition}

\newtheorem{cor}[thm]{Corollary}

\newtheorem{rmk}[thm]{Remark}
\newtheorem{prob}[thm]{Problem}

\DeclareMathOperator{\lcm}{lcm}

\newcommand{\F}{\mathbb{F}}

\newcommand{\C}{\mathbb{C}}
\newcommand{\HH}{\mathbb{H}}
\newcommand{\st}{{|}}
\newcommand{\rk}{{\rm rk}}

\newcommand{\fri}{{\mathrm{i}}}
\newcommand{\frj}{{\mathrm{j}}}
\newcommand{\frk}{{\mathrm{k}}}

\title[Nullstellensatz, Chevalley--Warning, and Finitesatz]{The Combinatorial Nullstellensatz, Chevalley-Warning Theorem, and weak Finitesatz in skew polynomial rings}
\author{Gil Alon}
\email{gilal@openu.ac.il}
\address{Department of Mathematics and Computer Science, The Open University of Israel}

\author{Angelot Behajaina}
\email{angelot.behajaina@univ-lille.fr}
\address{Univ. Lille, CNRS, UMR 8524, Laboratoire Paul Painlevé, F-59000 Lille, France}

\author{Elad Paran}
\email{paran@openu.ac.il}
\address{Department of Mathematics and Computer Science, The Open University of Israel}

\begin{document}

\begin{abstract}
    We study zeros of polynomials in the multivariate skew polynomial ring $D[x_1,\ldots,x_n; \sigma]$, where $\sigma$ is an automorphism of a division ring $D$. We prove a generalization of Noga Alon's celebrated Combinatorial Nullstellensatz for such polynomials. In the case where $D$ is a finite field, we prove skew analogues of the Chevalley--Warning theorem, Ax's Lemma, and the weak case of Terjanian's Finitesatz.
\end{abstract}

\maketitle

\section{Introduction}

Skew polynomial rings were first introduced by Noether and Schmeidler in \cite{NS20}, and their theoretical framework was established by Ore in his classical paper \cite{Ore33}. These rings have been extensively studied in the literature, both for their ring-theoretic properties (for example in \cite{Jac37,Coh63,Jat71,Ram84,GL94,SZ02,MK19}) and for their various applications (for example in \cite{IM94,BU09,She20}). A systematic study of zero sets of skew polynomials in one variable was carried out by Lam and Leroy in their papers \cite{Lam86,VL88,LL04,LaL08}. 
In the present work, building upon the works of Lam and Leroy, we study zeros of polynomials in the {\bf multivariate} skew polynomial ring $D[x_1,\ldots,x_n;\sigma]$, see Definition \ref{def:multi} below. Such rings belong in a broader family of \emph{iterated Ore extensions} $A[x_1;\sigma_1,\delta_1][x_2; \sigma_2,\delta_2]\dotsc[x_n;\sigma_n,\delta_n]$. The general theory of evaluation in such rings is developed in \cite{LN25}; see also the survey \cite[\S7]{LM25}. We will use the same notion of evaluation here.

We will prove several thematically related analogues of classical results concerning zeros of multivariate polynomials over fields: The Combinatorial Nullstellensatz of N. Alon, the Chevalley--Warning theorem, and the weak case of the Finitesatz of Terjanian. 

\subsection{The Combinatorial Nullstellensatz}

The celebrated Combinatorial Nullstellensatz of N. Alon \cite[Theorem 1.2]{Alon99} is now a classical result of algebraic combinatorics. It states that a non-zero multivariate polynomial $p \in K[x_1,\ldots,x_n]$ over a field $K$ takes a nonzero value on any sufficiently large grid in $K^n$. This theorem has numerous applications in various areas of combinatorics. The theorem was extended from fields to division rings by the third author \cite[Theorem 1.1]{Par23}. This Combinatorial Nullstellensatz over division rings has implications to the additive theory of division rings, see \cite[\S3,\S4]{Par23}. Here, we extend the Combinatorial Nullstellensatz further and prove the following generalization for the ring $D[x_1,\ldots,x_n;\sigma]$:

\begin{thm}[Skew Combinatorial Nullstellensatz]\label{comb_null_main} Let $p \in D[x_1,\ldots,x_n;\sigma]$ be of total degree $\deg(p)=\sum_{i=1}^n {k_i}$, where each $k_i$ is a non-negative integer, such that the coefficient of $x_1^{k_1}\cdots x_n^{k_n}$ in $p$ is non-zero. Let $A_1,\ldots,A_n$ be $\sigma$-algebraic subsets of $D$ such that $A_1 \times \dots \times A_n \subseteq D^{n,\sigma}$ and that $\rk_\sigma(A_i) > k_i$ for all $1 \leq i \leq n$. Then there is a point in $A_1 \times \dots \times A_n$ at which $p$ does not vanish. 
\end{thm}

Here, following Lam and Leroy, we say that a set $A \subseteq D$ is {\it $\sigma$-algebraic} if there exists a non-zero polynomial in $D[x;\sigma]$ that vanishes at all elements of $A$, see \S\ref{sec:prelim}. The space $D^{n,\sigma}$ is the \emph{$\sigma$-affine space} in $D$, and $\rk_\sigma(A_i)$ denotes the $\sigma$-rank of $A_i$, see Definition \ref{def:affine} and Definition \ref{def:rank} below. We say that a polynomial {\it vanishes} at a point of $D^{n,\sigma}$ if its value at that point is $0$.

In the special case where $\sigma$ is the identity automorphism, Theorem \ref{comb_null_main} recovers \cite[Theorem 1.1]{Par23}, and if in addition $D$ is a field, we recover Alon's original theorem \cite[Theorem 1.2]{Alon99}.

\subsection{The Chevalley--Warning theorem}

Let $f_1,\ldots,f_r \in F[x_1,\ldots,x_n]$ be $r$ polynomials in $n$ variables over a finite field $F$ of characteristic $p$ and  order $q$. Let $d_i$  denote the total degree of $f_i$, for each $1 \leq i \leq r$. The classical Chevalley--Warning theorem states that if $n > d_1+\dots+d_r$, then the number of common solutions in $F^n$ for $f_1,\ldots,f_r$ is divisible by $p$. 

Suppose now that $\sigma$ is an automorphism of $F$, and let $o(\sigma)$ denote its order. In \S\ref{sec:cheval} we prove the following theorem:

\begin{thm}[Skew Chevalley--Warning theorem]\label{thm:chev_intro}
Let $f_1,\dots,f_r \in F[x_1,\dots,x_n;\sigma]$ be polynomials such that $\deg(f_1)+\cdots+\deg(f_r)<n \cdot\big(\frac{q^{\frac{1}{o(\sigma)}}-1}{q-1}\big)$. Then the number of common zeros of $f_1,\ldots,f_r$ in the $\sigma$-affine space $F^{n,\sigma}$ is divisible by $p$.
\end{thm}

Note that in the special case where $\sigma$ is the identity, Theorem \ref{thm:chev_intro} recovers the usual Chevalley--Warning theorem. The proof of our result involves a variant of another classical result, Ax's Lemma, see Lemma \ref{lem:arev} below.

\subsection{The Finitesatz and the ideal of everywhere-vanishing polynomials}

Given a field $F$ and an ideal $J$ in the polynomial ring $F[x_1,\ldots,x_n]$, let $\mathscr{V}(J)$ denote the set of common zeros of $J$ in $F^n$, and let $\mathscr{I}(\mathscr{V}(J))$ denote the vanishing ideal of $\mathscr{V}(J)$. Describing this ideal is a fundamental question of algebraic geometry over $F$. In the case where $F$ is algebraically closed, $\mathscr{I}(\mathscr{V}(J))$ is the radical $\sqrt{J}$ of $J$, by Hilbert's Nullstellensatz. In the case where $F$ is a finite field of characteristic $p$ and order $q = p^m$, we have the ``Finitesatz" of Terjanian \cite{Ter66}, which states that $\mathscr{I}(\mathscr{V}(J)) =  J+\langle x_1^{q}-x_1,\dots,x_n^q-x_n \rangle$. In the special case where $\mathscr{V}(J) = \emptyset$  is the empty set we have the ``weak" Finitesatz, which states that $F[x_1,\ldots,x_n] = \mathscr{I}(\mathscr{V}(J)) =  J+\mathscr{I}(F^n)$, where $\mathscr{I}(F^n) = \langle x_1^{q}-x_1,\dots,x_n^q-x_n \rangle$ is the ideal of polynomials vanishing everywhere in $F^n$ (this corresponds to the classical weak Nullstellensatz, which states that if an ideal $J$ in $\C[x_1,\ldots,x_n]$ has an empty zero set, then $\C[x_1,\ldots,x_n] = J = J+ (0)$, where $(0)$ is the ideal of functions vanishing everywhere in $\C^n$).

Suppose now that $\sigma$ is an automorphism of $F$. Then $\sigma={\rm Frob}^k$, for a suitable $0 \leq k \leq m-1$, where ${\rm Frob}$ is the Frobenius automorphism of $F$. Given a left ideal $J$ in $F[x_1,\ldots,x_n;\sigma]$, let $\mathscr{V}(J)$ denote its zero set in $F^{n,\sigma}$, and let $\mathscr{I}(\mathscr{V}(J))$ denote the left ideal of polynomials vanishing at $\mathscr{V}(J)$. For $n = 1$, we prove that $\mathscr{I}(\mathscr{V}(J)) = J+F[x;\sigma]\cdot \left(x^{\frac{m}{\theta}(p^\theta-1)+1}-x\right)$, where $\F_{p^\theta}$ is the fixed field of $\sigma$ and $\theta = \gcd(m,k)$, see Theorem \ref{thm:skewonedim} below. 

For $n > 1$, we do not have a general description of $\mathscr{I}(\mathscr{V}(J))$ -- this seems a more difficult problem than its commutative counterpart. However, we are able to prove the ``weak" skew Finitesatz: If $\mathscr{V}(J)$ is empty, then $$F[x_1,\dots,x_n;\sigma]=\mathscr{I}(\mathscr{V}(J))=J+\mathscr{I}(F^{n,\sigma}),$$
where $\mathscr{I}(F^{n,\sigma})$ is the left ideal of polynomials in $F[x_1,\ldots,x_n;\sigma]$ which vanish everywhere in $F^{n,\sigma}$, see Theorem \ref{thm:skewfinsatz} below. We also give an explicit description of the ideal $\mathscr{I}(F^{n,\sigma})$, see Theorem \ref{thm:descrvanide}.

{\bf Acknowledgements.} The second author acknowledges the support of the CDP C2EMPI, as well as the French State under the France-2030 programme, the University of Lille, the Initiative of Excellence of the University of Lille, the European Metropolis of Lille for their funding and support of the R-CDP-24-004-C2EMPI project. He is also grateful for the support of a Technion fellowship, of an Open University of
Israel post-doctoral fellowship, and of the Israel Science Foundation (grant no. 353/21) during the completion of certain parts of this work.

\section{Preliminaries}\label{sec:prelim}

For the reader's convenience, we gather in this section some basic material on {\bf multivariate skew polynomials}. This notion generalizes the classical skew polynomials in one variable introduced by Ore \cite{Ore33}. For further reference, see for example \cite{Vos86}. 

Fix a division ring $D$ and an automorphism $\sigma$ of $D$. 

\subsection{The general multivariate case}\label{subsec:multi_variate}

Let $x_1,\dots,x_n$ ($n\geq 1$) denote $n$ variables. 
\begin{defn}\label{def:multi} The \emph{multivariate skew polynomial ring} $R=D[x_1,\ldots,x_n ; \sigma]$ consists of all formal finite sums 
$$\sum_{(k_1,\dots,k_n) \in \mathbb{N}^n}a_{k_1,\dots,k_n} x_1^{k_1}\cdots x_n^{k_n}$$ with coefficients $a_{k_1,\dots,k_n} \in D$, where addition is defined component-wise, and multiplication satisfies the following rules:
\begin{itemize}
\item $x_ix_j=x_jx_i$ for all $1\leq i,j \leq n$;
\item $x_i\cdot a=\sigma(a)x_i$ for all $1\leq i \leq n$ and all $a \in D$. 
\end{itemize}
\end{defn}
For every ${\bf a}=(a_1,\dots,a_n) \in D^n$, consider the {\bf left} ideal
$$\mathfrak{m}_{{\bf a}}=R(x_1-a_1)+\dots+R(x_n-a_n)
$$
of $R$, generated by $x_1-a_1,\dots,x_n-a_n$. Unlike in the commutative case, this ideal is not always proper. Proposition \ref{prop:maproper} below provides a necessary and sufficient condition for $\mathfrak{m}_{{\bf a}}$ to be proper. First, we have:

\begin{lem}\label{lem:condformanontri}
Let ${\bf a} \in D^n$. Assume that $\sigma(a_j)a_i \neq \sigma(a_i)a_j$ for some $1 \leq i \neq j \leq n$. Then $\mathfrak{m}_{{\bf a}}=R$.
\end{lem}
\begin{proof}
Note that
\begin{align*}
\left(x_j-\sigma(a_j) \right)\left(x_i-a_i\right)-\left(x_i-\sigma(a_i)\right)\left(x_j-a_j\right)
&=\sigma(a_j) a_i-\sigma(a_i) a_j \in D^\times \cap \mathfrak{m}_{{\bf a}}.
\end{align*}
Since $D^\times \subset R^\times$, it follows that $\mathfrak{m}_{{\bf a}}=R$.
\end{proof}

We shall mainly work on the following subset of $D^n$.

\begin{defn}\label{def:affine}
The \emph{$\sigma$-affine space} in $D^n$ is given by
$$
D^{n,\sigma}=\left\lbrace {\bf a} \in D^n \mid \sigma(a_j) a_i = \sigma(a_i) a_j\,\, \textrm{for all}\,\,1\leq i\neq j \leq n\right\rbrace.
$$ 
\end{defn}

Next, let us consider the evaluation of multivariate skew polynomials. Recall first the one-variable case. Let
\[
h=\sum_{k=0}^m c_kx^k\in D[x;\sigma]
\]
and let $a\in D$. Right division by $x-a$ gives unique $q\in D[x;\sigma]$ and $\ell\in D$ such that
\[
h=q(x-a)+\ell.
\]
The evaluation of $h$ at $a$ is defined to be $h(a)=\ell$; see, for example, \cite[\S2]{Lam86}.

Throughout this paper, for $d_1,\dots,d_m \in D$, we denote by $\prod_{i=1}^m d_i$ the product taken in the following order: $d_m d_{m-1} \cdots d_1$.
Moreover, the \emph{$k$-norm} ($k \geq 0$) of $a \in D$ is defined by
$$
{\rm N}^\sigma_k(a)=
\begin{cases}
1 & \textrm{if}\,\, k=0,\\
\prod_{i=0}^{k-1}\sigma^i(a)=\sigma^{k-1}(a)\cdots \sigma(a) a  & \textrm{if}\,\, k \geq 1.
\end{cases}
$$
Equivalently, by \cite[Lemma 2.4]{VL88}, the one-variable evaluation is given by
\[
h(a)=\sum_{k=0}^m c_k{\rm N}^\sigma_k(a).
\]

\begin{defn}[Ordered evaluation]\label{def:ordeval}
Put $R_0=D$ and $R_i=D[x_1,\ldots,x_i;\sigma]$ for $1\leq i\leq n$. For ${\bf a}=(a_1,\ldots,a_n)\in D^n$, define the additive subgroup
\[
\mathfrak{m}_{{\bf a}}^{\rm ord}=R_1(x_1-a_1)+\cdots+R_n(x_n-a_n)
\]
of $R$. Successive right division, first by $x_n-a_n$, then by $x_{n-1}-a_{n-1}$, and so on, gives for every $f\in R$ a unique scalar $f^{\rm ord}({\bf a})\in D$ such that
\[
f-f^{\rm ord}({\bf a})\in\mathfrak{m}_{{\bf a}}^{\rm ord}.
\]
We call this scalar the \emph{ordered evaluation} of $f$ at ${\bf a}$, for the order $x_n,x_{n-1},\ldots,x_1$.
\end{defn}
This is the ordered evaluation for iterated Ore extensions considered in \cite[Definition~2.3]{LN25}; see also \cite[Definition~7.2]{LM25}. In general, a different ordering of the variables may give a different value. As we shall see below, however, at points of $D^{n,\sigma}$ the value is independent of the order.

\begin{lem}\label{lem:ordevalformula}
For ${\bf a}\in D^n$ and $k_1,\ldots,k_n\geq 0$, the ordered evaluation in Definition~\ref{def:ordeval} is given by
\begin{equation}\label{eq:defeval}
\left[x_1^{k_1}\cdots x_{n}^{k_n}\right]^{\rm ord}({\bf a})=\prod_{i=1}^n \sigma^{\sum_{j=1}^{i-1}k_j}({\rm N}^\sigma_{k_i}(a_i))= \sigma^{\sum_{j=1}^{n-1}k_j}({\rm N}^\sigma_{k_n}(a_n))\cdots \sigma^{k_1}({\rm N}^\sigma_{k_2}(a_2)) {\rm N}^{\sigma}_{k_1}(a_1).
\end{equation}
\end{lem}
\begin{proof}
We argue by induction on $n$. For $n=1$, the assertion is the one-variable norm formula recalled above. Suppose that $n>1$ and put $s=k_1+\cdots+k_{n-1}$. Right division first by $x_n-a_n$ replaces $x_n^{k_n}$ by ${\rm N}^{\sigma}_{k_n}(a_n)$, so the first remainder is
\[
x_1^{k_1}\cdots x_{n-1}^{k_{n-1}}{\rm N}^{\sigma}_{k_n}(a_n)
=\sigma^s\big({\rm N}^{\sigma}_{k_n}(a_n)\big)x_1^{k_1}\cdots x_{n-1}^{k_{n-1}}.
\]
Applying the induction hypothesis to the remaining variables gives exactly \eqref{eq:defeval}.
\end{proof}

\begin{lem}\label{lem:annulfcn} Let ${\bf a} \in D^{n,\sigma}$. Then
\[
\mathfrak{m}_{{\bf a}}^{\rm ord}=\mathfrak{m}_{{\bf a}}.
\]
In particular, any polynomial $g \in \mathfrak{m}_{{\bf a}}$ vanishes at ${\bf a}$.
\end{lem}
\begin{proof}
Clearly $\mathfrak{m}_{{\bf a}}^{\rm ord}\subseteq\mathfrak{m}_{{\bf a}}$. We claim that $\mathfrak{m}_{{\bf a}}^{\rm ord}$ is a left ideal. It is plainly stable under left multiplication by elements of $D$. Let $h\in R_i$. If $j\leq i$, then $x_jh(x_i-a_i)\in R_i(x_i-a_i)$. If $j>i$, write $x_jh=h^\sigma x_j$, where $h^\sigma\in R_i$ is obtained by applying $\sigma$ to the coefficients of $h$. Since ${\bf a}\in D^{n,\sigma}$,
\begin{equation}\label{eq:goodpointidentity}
x_j(x_i-a_i)=\sigma(a_j)(x_i-a_i)+(x_i-\sigma(a_i))(x_j-a_j).
\end{equation}
It follows that $x_jh(x_i-a_i)\in\mathfrak{m}_{{\bf a}}^{\rm ord}$. Thus $\mathfrak{m}_{{\bf a}}^{\rm ord}$ is a left ideal containing all the elements $x_i-a_i$, and hence $\mathfrak{m}_{{\bf a}}\subseteq\mathfrak{m}_{{\bf a}}^{\rm ord}$. This proves the equality. If $g\in\mathfrak{m}_{{\bf a}}$, then $g\in\mathfrak{m}_{{\bf a}}^{\rm ord}$, so its ordered evaluation is $0$, that is, $g^{\rm{ord}}({\bf a})=0$.
\end{proof}

\begin{rmk}
In the terminology of \cite[Definition~7.6]{LM25}, the points of $D^{n,\sigma}$ are \emph{good points}; see also \cite[Theorem~7.7]{LM25} and \cite[Theorem~3.9]{LN25}.
\end{rmk}

\begin{prop}\label{prop:maproper}
Let ${\bf a} \in D^n$. Then $\mathfrak{m}_{{\bf a}}$ is a proper left ideal of $R$ if and only if ${\bf a} \in D^{n,\sigma}$. In this case, $\mathfrak{m}_{{\bf a}}$ is maximal. 
\end{prop}
\begin{proof} Assume that ${\mathfrak{m}}_{{\bf a}}$ is proper. Then, by Lemma \ref{lem:condformanontri}, it follows that ${\bf a} \in D^{n,\sigma}$. Conversely, assume that ${\bf a} \in D^{n,\sigma}$. Since $1$ does not vanish at ${\bf a}$, Lemma \ref{lem:annulfcn} implies that $1 \notin \mathfrak{m}_{{\bf a}}$, and hence $\mathfrak{m}_{{\bf a}}$ is proper. This completes the proof of the first statement. 

For the second statement, assume ${\bf a}\in D^{n,\sigma}$ and suppose $g \notin \mathfrak{m}_{{\bf a}}$. By ordered evaluation and Lemma \ref{lem:annulfcn}, we can write $g=f+\ell$ for some $f \in \mathfrak{m}_{{\bf a}}$ and $\ell \in D$. Since $g\notin\mathfrak{m}_{{\bf a}}$, we have $\ell\neq 0$. Hence, $1=-\ell^{-1}f + \ell^{-1}g \in \mathfrak{m}_{{\bf a}}+R \cdot g$. Thus, $\mathfrak{m}_{{\bf a}}+R \cdot g=R$, so $\mathfrak{m}_{{\bf a}}$ is maximal.
\end{proof}

Since the variables commute and are governed by the same automorphism, the proof of Lemma~\ref{lem:annulfcn} applies verbatim after any permutation of the variables. Hence, from Lemma~\ref{lem:annulfcn} and Definition~\ref{def:ordeval}, we immediately get:

\begin{cor}\label{cor:evalformon}
Let $f \in R$ and ${\bf a} \in D^{n,\sigma}$. For any ordering of the variables, the corresponding ordered evaluation is the unique $\ell \in D$ such that $f-\ell \in \mathfrak{m}_{{\bf a}}$. In particular, the ordered evaluation is independent of the order.
\end{cor}

Henceforth, for $f \in R$ and ${\bf a} \in D^{n,\sigma}$, we denote this common ordered evaluation by $f({\bf a})$.

\begin{defn}
Let ${\bf a}=(a_1,\dots,a_n) \in D^n$ and let $b \in D^\times$. The \emph{$\sigma$-conjugate of ${\bf a}$ by $b$} is given by
$${\bf a}^b=(\sigma(b) a_1 b^{-1},\dots, \sigma(b) a_n b^{-1}).
$$
\end{defn}
We shall use the fact that $D^{n,\sigma}$ is stable under $\sigma$-conjugation. Indeed, if ${\bf a}\in D^{n,\sigma}$, then
$$
\sigma(a_j^b)a_i^b=\sigma^2(b)\sigma(a_j)a_i b^{-1}=\sigma^2(b)\sigma(a_i)a_j b^{-1}=\sigma(a_i^b)a_j^b.
$$

As in the one variable case \cite{VL88}, we have the following product formula.
\begin{lem}[Product formula]\label{lem:prodformskewmult}
Let $f,g \in R$ and let ${\bf a} \in D^{n,\sigma}$. Then
$$
(f \cdot g)({\bf a})=
\begin{cases}
0 & \textrm{if}\,\, g({\bf a})=0,\\
f({\bf a}^{g({\bf a})})g({\bf a}) &\textrm{if}\,\, g({\bf a}) \neq 0. 
\end{cases}
$$
\end{lem}
\begin{proof}
We adapt the proof of the one variable case \cite[Theorem 2.7]{VL88} to the multivariate setting. If $g({\bf a})=0$, then $g \in \mathfrak{m}_{{\bf a}}$ by Corollary \ref{cor:evalformon}, so $f \cdot g \in \mathfrak{m}_{{\bf a}}$, and hence $(f\cdot g)({\bf a})=0$. Now, assume that $g({\bf a}) \neq 0$. Set $c=g({\bf a})$ and ${\bf b}={\bf a}^{c}$. By Corollary \ref{cor:evalformon}, write
$$
g=\sum_{i=1}^n A_i\cdot (x_i-a_i)+c
$$
and
$$
f=\sum_{i=1}^n B_i\cdot (x_i-b_i)+f({\bf b}),
$$ where $A_1,\dots,A_n, B_1,\dots,B_n \in R$. Note that
$$
(x_i-b_i)\cdot c=(x_i-\sigma(c) a_i c^{-1})c=\sigma(c)(x_i-a_i),
$$ for all $1 \leq i \leq n$.
Therefore
\begin{align*}
f\cdot g&=\sum_{i=1}^nf \cdot A_i \cdot (x_i-a_i)+f \cdot c\\
&=\sum_{i=1}^n f\cdot A_i \cdot (x_i-a_i)+\sum_{i=1}^n B_i \cdot (x_i-b_i)c +f({\bf b})c\\
   &=\underbrace{\sum_{i=1}^n f\cdot A_i \cdot (x_i-a_i)+\sum_{i=1}^n B_i \cdot \sigma(c) \cdot (x_i-a_i)}_{\in \mathfrak{m}_{{\bf a}}} +f({\bf b})c.
\end{align*}
Evaluating both sides at ${\bf a}$, we obtain
$$
(f \cdot g)({\bf a})=f({\bf b})c=f({\bf a}^{g({\bf a})})g({\bf a}).
$$
\end{proof}

\subsection{The one variable case} In this part we focus on the one variable case. 
\begin{defn}
Let $0 \neq f,g \in D[x;\sigma]$. The \emph{ left-hand least common multiple} of $f,g$, denoted by $\lcm(f,g)$, is the monic polynomial of minimal degree that is divisible from the right by both $f$ and $g$.
\end{defn}
\begin{rmk}
 The polynomial $\lcm(f,g)$ always exists and is uniquely determined by $f$ and $g$ \cite[p.~485]{Ore33}. Moreover, if $f = x-a,g = x-b$ are monic linear polynomials, then $\lcm(f,g)$ is the monic polynomial of smallest degree that vanishes at both $a$ and $b$ (since a polynomial $p \in D[x,\sigma]$ is divisible by $x-a$ from the right if and only if $p(a) = 0$). 
\end{rmk}
More generally, 
\begin{defn}
Let $S \subseteq D[x;\sigma]$. Assume that there exists a non-zero polynomial that is right-hand divisible by all polynomials in $S$. Then such a polynomial of minimal degree\footnote{This polynomial is uniquely determined by $S$.} is called the \emph{left-hand least common multiple} of the elements of $S$, and is denoted by $\lcm(S)$. 
\end{defn}

\begin{lem}\label{union} Let $S \subseteq D[x;\sigma]$, and suppose that $g = \lcm(S)$ exists. A polynomial $f \in D[x;\sigma]$ is right-hand divisible by all polynomials in $S$ if and only if $f$ is right-hand divisible by $g$. 
\end{lem} 
\begin{proof} 
Assume that $f$ is right-hand divisible by $g$. Then, by the definition of $g$, the polynomial $f$ is right-hand divisible by all polynomials in $S$.

Conversely, assume that $f$ is right-hand divisible by all polynomials in $S$. Then, by right-hand division with remainder we may write $f = qg+r$ with $\deg(r) < \deg(g)$. Hence $r = f-qg$ is right-hand divisible by all polynomials in $S$, and by the minimality of the degree of $g$ we must have $r = 0$. Therefore $f$ is right-hand divisible by $g$. 
\end{proof}

Following \cite[\S2]{VL88}, we define:
\begin{defn}
Let $A$ be a subset of $D$. We say that $A$ is \emph{ $\sigma$-algebraic}, if there exists a non-zero polynomial in $D[x;\sigma]$ that vanishes at all points of $A$. Equivalently, $A$ is $\sigma$-algebraic if $\lcm\{x-a \st a \in A\}$ exists. 
\end{defn}
\begin{rmk}
Every finite set in $D$ is $\sigma$-algebraic, but infinite $\sigma$-algebraic sets are possible. For example, if $D = \HH$ is the real quaternion algebra and $\sigma$ is the identity automorphism, then the set $$\{a \fri + b \frj + c \frk \st a,b,c \in \mathbb{R}, a^2+b^2+c^2 = 1\}$$ is $\sigma$-algebraic, with minimal polynomial $x^2+1$. Or, if $D = \C$ is the field of complex numbers and $\sigma$ is the usual complex conjugation, then the set 
$$\{z \in \C \st |z| = 1\}$$
is $\sigma$-algebraic, with minimal polynomial $x^2-1$.
\end{rmk}

\begin{defn}[Rank of a $\sigma$-algebraic set]\label{def:rank} Let $A$ be a $\sigma$-algebraic subset of $D$. The polynomial $\lcm\{x-a \st a \in A\} \in D[x;\sigma]$ is  called the \emph{ $\sigma$-minimal polynomial of $A$} and we shall denote it by $f_{A,\sigma}$. We shall call the \emph{degree} of $f_{A,\sigma}$ the \emph{$\sigma$-rank} of the set $A$, and denote it by $\rk_\sigma(A)$.
\end{defn}

The theory of $\sigma$-algebraic sets and their ranks was developed in \cite{Lam86}, \cite{LL04} and \cite{LaL08} (in greater generality, in the context of skew polynomial rings with an endomorphism and a derivation), but we shall not need any further results from there here. 
\begin{lem}\label{divides} Let $A$ be a non-empty $\sigma$-algebraic subset of $D$. Let $a \in A$. The polynomial $\big(\lcm(x-b^{b-a} \mid b \in A\setminus\{a\})\big)\cdot(x-a)$ is right-hand divisible in $D[x;\sigma]$ by $\lcm\{x-b \mid b \in A\}$. 
\end{lem} 
\begin{proof} Let $h = \lcm\{x-b^{b-a} \mid b \in A \setminus\{a\}\}$. By Lemma \ref{union}, we must show that $g = h\cdot(x-a)$ is right-hand divisible by $x-b$ for all $b \in A$. For $b = a$ this is evident. For a given $b \in A\setminus \{a\}$, since $h$ is right-hand divisible by $x-b^{b-a}$, it follows that $g$ is right-hand divisible by $p = (x-b^{b-a})(x-a)$. By the product formula (Lemma \ref{lem:prodformskewmult} or \cite[Theorem 2.7]{VL88}), we have $p(b) = (b^{b-a}-b^{b-a})(b-a) = 0$, hence $p$ is divisible by $x-b$, and hence so is $g$. \end{proof}

\section{Skew Combinatorial Nullstellensatz} \label{sec_central}

In this section, we establish the {\bf skew Combinatorial Nullstellensatz}. For that, let us fix a division ring $D$ and an automorphism $\sigma$ of $D$. 

\begin{thm}[Skew Combinatorial Nullstellensatz]\label{new2} Let $p \in D[x_1,\ldots,x_n;\sigma]$ be of total degree $\deg(p)=\sum_{i=1}^n {k_i}$, where each $k_i$ is a non-negative integer, such that the coefficient of $x_1^{k_1}\cdots x_n^{k_n}$ in $p$ is non-zero. Let $A_1,\ldots,A_n$ be $\sigma$-algebraic subsets of $D$ such that $A_1 \times \dots \times A_n \subseteq D^{n,\sigma}$ and that $\rk_\sigma(A_i) > k_i$ for all $1 \leq i \leq n$. Then there is a point in $A_1 \times \dots \times A_n$ at which $p$ does not vanish. 
\end{thm} 
\begin{proof}
We prove the theorem by induction on $\deg(p)$. If $\deg(p) = 0$, then $p$ is a non-zero constant in $D$, and the assertion holds trivially. 

Now suppose that $\deg(p) > 0$ and that we have proven the theorem for all polynomials of degree smaller than $\deg(p)$. Assume to the contrary that $p$ vanishes on $A_1 \times \dots \times A_n$. By Corollary \ref{cor:evalformon}, we may relabel the variables and assume without loss of generality that $k_1 > 0$. Choose $a_1 \in A_1$ and apply right-hand division with remainder to write
$$
p=q\cdot (x_1-a_1)+r,
$$
where $r \in D[x_2,\ldots,x_n;\sigma]$. Since in $p$ there appears a monomial of the form $\lambda \cdot x_1^{k_1}\cdots x_n^{k_n}$, it follows that in $q$ there appears a monomial of the form $\lambda \cdot x_1^{k_1-1}\cdots x_n^{k_n}$, and clearly $\deg(q) = \deg(p)-1$.  

Given a point ${\bf a}=(a_1,a_2,\ldots,a_n) \in \{a_1\} \times A_2 \times \dots \times A_n$, evaluate first in the variable $x_1$. By Corollary \ref{cor:evalformon}, this gives the same value as \eqref{eq:defeval}, and hence $r(a_2,\ldots,a_n)=p({\bf a})=0$. Thus $r$ vanishes on $A_2 \times \dots \times A_n$. In particular, for any point ${\bf b} \in (A_1 \setminus\{a_1\}) \times A_2 \times \dots \times A_n$, we have $r({\bf b}) = 0$, and therefore
\begin{equation}\label{eq:qxavan}
\big( q(x_1-a_1) \big) ({\bf b}) = p({\bf b}) - r({\bf b}) = 0.
\end{equation}

Fix ${\bf b}=(b_1,a_2,\dots,a_n) \in (A_1 \setminus\{a_1\}) \times A_2 \times \dots \times A_n$, and put $c=b_1-a_1$. By the product formula and \eqref{eq:qxavan},
$$
0=\big(q(x_1-a_1)\big)({\bf b})=q({\bf b}^{c})c,
$$
so $q({\bf b}^{c})=0$. Moreover, for $2 \leq i \leq n$, since both $(b_1,a_2,\ldots,a_n)$ and $(a_1,a_2,\ldots,a_n)$ belong to $D^{n,\sigma}$, we have
\begin{align}\label{eq:comaibmoina}
a_i^{b_1-a_1}
&= (\sigma(b_1)-\sigma(a_1))a_i(b_1-a_1)^{-1} \nonumber\\
&=\sigma(a_i)(b_1-a_1)(b_1-a_1)^{-1}=\sigma(a_i).
\end{align}
Thus
$$
{\bf b}^{c}=(b_1^{b_1-a_1},a_2^\sigma,\ldots,a_n^\sigma)\in D^{n,\sigma},
$$
where the final inclusion follows from the stability of $D^{n,\sigma}$ under $\sigma$-conjugation.

 Set $B_1 = \{b_1^{b_1-a_1} \mid b_1 \in A_1 \setminus \{a_1\} \}$. We have thus shown that $q$ vanishes on the set $B_1 \times A^\sigma_2 \times \cdots \times A^\sigma_n \subseteq D^{n,\sigma}$. Note that for each $2 \leq i \leq n$, the set $A_i^\sigma$ is $\sigma$-algebraic with $\rk_\sigma(A_i^\sigma) = \rk_\sigma(A_i)$. Indeed, if $f_i$ is the $\sigma$-minimal polynomial of $A_i$ then $f_i^\sigma$ is the minimal polynomial of $A_i^\sigma$. Now, consider the polynomial

$$\big(\lcm\{x_1-b_1^{b_1-a_1} \mid b_1 \in A_1\setminus\{a_1\}\}\big)\cdot(x_1-a_1) = \big(\lcm\{x_1-c_1 \mid c_1 \in B_1\}\big)\cdot(x_1-a_1).$$
By Lemma \ref{divides}, this polynomial is right-hand divisible in $D[x_1;\sigma]$ by $\lcm\{x_1-b_1 \mid b_1 \in A_1\}$. By our assumptions, the degree (which is $\rk_\sigma(A_1)$) of the latter polynomial is larger than $k_1$, hence 
$$\deg\big(\lcm\{x_1-c_1 \mid c_1 \in B_1\}\big) + 1 > k_1,
$$ so $\rk_\sigma(B_1) = \deg\big(\lcm\{x_1-c_1 \st c_1 \in B_1\}\big) > k_1-1$. Since $\deg(q) =\deg(p)-1 < \deg(p)$, the polynomial $q$ vanishes on $B_1 \times A_2^\sigma \times \dots \times A_n^\sigma$, and in $q$ there appears the monomial $\lambda x_1^{k_1-1}\cdot x_2^{k_2} \cdots x_n^{k_n}$, we get a contradiction with the induction hypothesis. 

Consequently, there is a point in $A_1\times \dots \times A_n$ at which $p$ does not vanish. This completes the inductive proof of the theorem. \end{proof}

\section{Skew Chevalley--Warning theorem}\label{sec:cheval}

In this section, we establish the {\bf skew Chevalley--Warning theorem}. Let $F=\mathbb{F}_q$ be a finite field with $q=p^m$ ($m \geq 1$) a prime power, and let $\sigma$ be an automorphism of $F$. Note that $\sigma={\rm Frob}^k$, for some $0 \leq k \leq m-1$, where ${\rm Frob}$ denotes the Frobenius automorphism of $F$.  Denote the fixed subfield of $F$ under $\sigma$ by $K$, that is, $K=F^\sigma$. Let $\theta=\gcd(k,m)$. Since the order of $\sigma$ is $o(\sigma)=\frac{m}{\theta}$, it follows that $|K|=q^{\frac{1}{o(\sigma)}}=p^{\theta}$. Now, consider the set 
$$\mathscr{W}_F=\{\sigma(a)a^{-1} \mid a \in F^*\}=\{a^{p^k-1}\mid a \in F^*\}.$$ 
For each $\lambda \in \mathscr{W}_F$, choose an element $\omega_\lambda \in F$ such that $\omega_\lambda^{p^k-1}=\lambda$, and let 
$$\mathfrak{S}_\lambda=\{a \in F \mid \sigma(a)=\lambda a\}=\{a \in F \mid a^{p^k-1}=\lambda\} \cup \{0\}=\omega_\lambda K.$$
{ In the following, we assume that the evaluation of any polynomial $g \in F[x_1,\dots,x_n;{\rm id}]$ is always performed in the classical sense.}
\begin{lem}\label{lem:carfcn}
Let $n \geq 1$. Then:
\begin{enumerate}
\item $|\mathscr{W}_F|=\frac{q-1}{q^{\frac{1}{o(\sigma)}}-1}$; \label{eq:cardWf}
\item $F^{n,\sigma}=\bigcup_{\lambda \in \mathscr{W}_F}\mathfrak{S}_\lambda ^n =\cup_{\lambda \in \mathscr{W}_F}\omega_\lambda K^n$;\label{eq:disunfcn}
\item and $|F^{n,\sigma}|=\frac{(q-1)\left(q^{\frac{n}{o(\sigma)}}-1\right)+q^{\frac{1}{o(\sigma)}}-1}{q^{\frac{1}{o(\sigma)}}-1}$. In particular $|F^{n,\sigma}| \equiv 0\pmod{p}$.
\end{enumerate}
\end{lem}
\begin{proof}
The first statement follows from the identity $\gcd(p^k-1,p^m-1)=p^{\gcd(k,m)}-1=q^{\frac{1}{o(\sigma)}}-1$, and from the fact that $\mathscr{W}_F$ is the kernel of the map given by raising to the power of $\frac{q-1}{q^{\frac{1}{o(\sigma)}}-1}$. The second statement is straightforward. For the third statement: By \eqref{eq:disunfcn}, we have 
$$F^{n,\sigma}=\bigcup_{\lambda \in \mathscr{W}_F}(\mathfrak{S}_\lambda )^n=\bigsqcup_{\lambda \in \mathscr{W}_F}\left(\omega_\lambda K^n \setminus \{{\bf 0}\}\right) \bigsqcup \{{\bf 0}\}.$$ Thus 
\begin{align*}
|F^{n,\sigma}|&=\sum_{\lambda \in \mathscr{W}_F}|\omega_\lambda K^n\setminus \{{\bf 0}\}|+1\\
       &=\sum_{\lambda \in \mathscr{W}_F}\left(q^{\frac{n}{o(\sigma)}}-1\right)+1\\
       &=
\frac{q-1}{q^{\frac{1}{o(\sigma)}}-1} \left(q^{\frac{n}{o(\sigma)}}-1\right)+1 \quad \quad \quad \quad \textrm{by \eqref{eq:cardWf}}\\
&=\frac{(q-1)\left(q^{\frac{n}{o(\sigma)}}-1\right)+q^{\frac{1}{o(\sigma)}}-1}{q^{\frac{1}{o(\sigma)}}-1}.
\end{align*}
\end{proof}
\noindent
For $f=\sum_{(k_1,\dots,k_n)} a_{k_1,\dots,k_n}x_1^{k_1}\cdots x_n^{k_n} \in F[x_1,\dots,x_n;\sigma]$ and $\lambda \in \mathscr{W}_F$, let $f^\lambda$ be the polynomial in $F[x_1,\dots,x_n;{\rm id}]$ defined by:
\begin{align*}
f^{\lambda}(x_1,\dots,x_n)&=\sum_{(k_1,\dots,k_n)}a_{k_1,\dots,k_n}\left[x_1^{k_1}\cdots  x_n^{k_n}\right]\overbrace{\left(\omega_\lambda,\dots,\omega_\lambda\right)}^{n\,\,\textrm{times}}x_1^{k_1}\cdots  x_n^{k_n}\\
&=\sum_{(k_1,\dots,k_n)}a_{k_1,\dots,k_n}\prod_{i=1}^n \sigma^{\sum_{j=1}^{i-1}k_j}({\rm N}^\sigma_{k_i}(\omega_\lambda))x_1^{k_1}\cdots  x_n^{k_n}.
\end{align*}
\begin{lem}\label{lem:resttoflam}
Let $f \in F[x_1,\dots,x_n;\sigma]$ and let $\lambda \in \mathscr{W}_F$. Then 
$$
f(\omega_\lambda {\bf a})=f^{\lambda}({\bf a}),
$$ for all ${\bf a} \in K^n$.
\end{lem}
\begin{proof}
By linearity, we may assume that $f=x_1^{k_1}\cdots x_n^{k_n}$ with $k_1,\dots,k_n \geq 0$. For ${\bf a} \in K^n$, we have
\begin{align*}
f(\omega_\lambda {\bf a})&=\prod_{i=1}^n \sigma^{\sum_{j=1}^{i-1}k_j}({\rm N}^\sigma_{k_i}(\omega_\lambda a_i))\\
&=\prod_{i=1}^n \sigma^{\sum_{j=1}^{i-1}k_j}({\rm N}^\sigma_{k_i}(\omega_\lambda))\prod_{i=1}^n \sigma^{\sum_{j=1}^{i-1}k_j}({\rm N}^\sigma_{k_i}(a_i))\quad\quad\textrm{since}\,\, \sigma\,\,\textrm{is a morphism}\\
&=\left(\prod_{i=1}^n \sigma^{\sum_{j=1}^{i-1}k_j}({\rm N}^\sigma_{k_i}(\omega_\lambda))\right)  a_1^{k_1}\cdots a_n^{k_n}\quad\quad\textrm{since}\,\, {\bf a} \in K^n\\
&=f^\lambda({\bf a}).
\end{align*}
\end{proof}

The following Lemma is a variant of the so-called Ax's Lemma \cite[Lemma 1.2]{CT24}.

\begin{lem}\label{lem:arev}
Let $g \in F[x_1,\dots,x_n;{\rm id}]$ be such that $\deg(g) < n\left(q^{\frac{1}{o(\sigma)}}-1\right)=n(|K|-1)$. Then 
$$
\sum_{{\bf a} \in K^n}g({\bf a})=0.
$$
\end{lem}
\begin{proof}
By linearity, we may assume that $g=x_1^{k_1}\cdots x_n^{k_n}$. Since $\deg(g) < n\left(|K|-1\right)$, there exists $1 \leq i_0 \leq n$ such that $k_{i_0} < |K|-1$.  Hence
$$
\sum_{{\bf a} \in K^n}g({\bf a})=\sum_{{\bf a} \in K^n}a_1^{k_1}\cdots a_n^{k_n}=\left[\prod_{i\neq i_0}\left( \sum_{a \in K}a^{k_i} \right)\right] \cdot \underbrace{\left(\sum_{a \in K}a^{k_{i_0}}\right)}_{=0}=0.
$$
\end{proof}
In the classical setting, the Chevalley--Warning theorem can be proved using Ax's lemma. The following result is a skew analogue of Ax's lemma. However, we were unable to use it to directly prove the skew Chevalley--Warning theorem.
\begin{prop}[Skew Ax's lemma]\label{prop:axlemma}
Let $f \in F[x_1,\dots,x_n;\sigma]$ be such that $\deg(f)<n\left(q^{\frac{1}{o(\sigma)}}-1\right)$. Then
$$
\sum_{{\bf a} \in F^{n,\sigma}}f({\bf a})=0.
$$
\end{prop}
\begin{proof}
By linearity, we may assume that $f=x_1^{k_1}\cdots x_n^{k_n}$ with $k_1,\dots,k_n \geq 0$. Then 
\begin{align*}
\sum_{{\bf a} \in F^{n,\sigma}}f({\bf a})&=\sum_{{\bf a} \in \bigsqcup_{\lambda \in \mathscr{W}_F}\left(\omega_\lambda K^n \setminus \{{\bf 0}\}\right) \bigsqcup \{{\bf 0}\}}f({\bf a})\\
&=f({\bf 0})+\sum_{\lambda \in \mathscr{W}_F}\sum_{{\bf a} \in K^n \setminus \{{\bf 0}\}}f(\omega_\lambda {\bf a })\\
&=f({\bf 0})+\sum_{\lambda \in \mathscr{W}_F}\sum_{{\bf a} \in K^n \setminus \{{\bf 0}\}}f^\lambda({\bf a })\quad\quad \textrm{by Lemma \ref{lem:resttoflam}}\\
&=f({\bf 0})+\sum_{\lambda \in \mathscr{W}_F}\sum_{{\bf a} \in K^n}f^\lambda({\bf a})-\sum_{\lambda \in \mathscr{W}_F}f^\lambda({\bf 0})\\
&=\sum_{\lambda \in \mathscr{W}_F}\sum_{{\bf a} \in K^n}f^\lambda({\bf a})+(\underbrace{-|\mathscr{W}_F|+1}_{\equiv 0 \pmod{p}})f({\bf 0})\,\,\,\,\,\,\,\,\,\,\,\textrm{by Lemma \ref{lem:resttoflam}}\\
&=0 \quad \quad\textrm{by Lemma \ref{lem:arev} since $\deg(f^\lambda)<n\left(q^{\frac{1}{o(\sigma)}}-1\right)$}.
\end{align*}
\end{proof}
\noindent
For $f_1,\dots,f_r \in F[x_1,\dots,x_n;\sigma]$, define 
$$
\mathscr{V}(f_1,\dots,f_r)=\left\lbrace{\bf a} \in F^{n,\sigma} \mid f_1({\bf a})=f_2({\bf a})=\dots=f_r({\bf a})=0\right\rbrace.
$$

\begin{thm}[Skew Chevalley--Warning theorem]\label{thm:skewchevwar}
Let $f_1,\dots,f_r \in F[x_1,\dots,x_n;\sigma]$ be polynomials such that $\deg(f_1)+\cdots+\deg(f_r)<n \frac{q^{\frac{1}{o(\sigma)}}-1}{q-1}$. Then $|\mathscr{V}(f_1,\dots,f_r)|\equiv 0 \pmod{p}$. 
\end{thm}
\begin{proof}
Let $$P_f=\sum_{\lambda \in \mathscr{W}_F}\left(1- (f^\lambda_1)^{q-1} \right) \left(1- (f^\lambda_2)^{q-1} \right) \cdots \left(1- (f^\lambda_r)^{q-1} \right) \in F[x_1,\dots,x_n;{\rm id}].$$
Then
\begin{align*}
\sum_{{\bf a} \in K^n} P_f({\bf a})&=\sum_{{\bf a} \in K^n}\sum_{\lambda \in \mathscr{W}_F}\prod_{i=1}^r\left(1-\left(f^\lambda_i({\bf a})\right)^{q-1}\right)\\
&=\sum_{\lambda \in \mathscr{W}_F}\sum_{{\bf a}\in K^n}\prod_{i=1}^r\left(1-\left(f_i(\omega_\lambda {\bf a})\right)^{q-1}\right)\quad\quad\textrm{by Lemma \ref{lem:resttoflam}}.
\end{align*}
Since $F^{n,\sigma}\setminus\{0\}$ is the disjoint union of the sets $\omega_\lambda K^n\setminus\{0\}$, we have
\begin{align*}
\sum_{{\bf a} \in K^n} P_f({\bf a})&=\underbrace{\sum_{{\bf b} \in F^{n,\sigma}}\prod_{i=1}^r\left(1-\left(f_i({\bf b})\right)^{q-1}\right)}_{=|\mathscr{V}(f_1,\dots,f_r)|}-\prod_{i=1}^r\left(1-\left(f_i({\bf 0})\right)^{q-1}\right)+
\sum_{\lambda \in \mathscr{W}_F}\prod_{i=1}^r\left(1-\left(f_i({\bf 0})\right)^{q-1}\right)\\
&=|\mathscr{V}(f_1,\dots,f_r)|+(\underbrace{|\mathscr{W}_F|-1}_{\equiv 0 \pmod{p}})\prod_{i=1}^r\left(1-\left(f_i({\bf 0})\right)^{q-1}\right)\\
&=|\mathscr{V}(f_1,\dots,f_r)|
\end{align*}
Since $\deg(P_f)<n\left(q^{\frac{1}{o(\sigma)}}-1\right)$, by Lemma \ref{lem:arev}, we obtain
$$
|\mathscr{V}(f_1,\dots,f_r)|= \sum_{{\bf a} \in K^n}P_f({\bf a}) = 0.
$$
Thus $|\mathscr{V}(f_1,\dots,f_r)|\equiv 0 \pmod{p}$.
\end{proof}
\begin{cor}\label{cor:effnull}
Let $f_1,\dots,f_r \in F[x_1,\dots,x_n;\sigma]$ be nonconstant homogeneous polynomials such that $\deg(f_1)+\dots+\deg(f_r)<n \frac{q^{\frac{1}{o(\sigma)}}-1}{q-1}$. Then $|\mathscr{V}(f_1,\dots,f_r)|\geq p$. 
\end{cor}
\begin{proof}
Since ${\bf 0} \in \mathscr{V}(f_1,\dots,f_r)$, Theorem \ref{thm:skewchevwar} implies that $|\mathscr{V}(f_1,\dots,f_r)|\geq p$. 
\end{proof}
\section{Skew Finitesatz}\label{sec:skewfinitesatz}

In this section, we establish the results related to the skew Finitesatz.

We use the same notation as in Section~\ref{sec:cheval}. Additionally, for any subset $J\subset F[x_1,\dots,x_n;\sigma]$ and $W \subset F^{n,\sigma}$, define
$$
\mathscr{V}(J)=\left\lbrace {\bf a} \in F^{n,\sigma} \mid f({\bf a})=0\,\,\textrm{for all}\,\, f \in J \right\rbrace
$$ and
$$
\mathscr{I}(W)=\left\lbrace f \in F[x_1,\dots,x_n;\sigma] \mid f({\bf a})=0\,\,\textrm{for all}\,\, {\bf a} \in W \right\rbrace.
$$
\subsection{Weak Skew Finitesatz}

In this part we prove a version of the skew Finitesatz for ideals with an empty zero locus -- an analogue of Hilbert's ``weak" Nullstellensatz. 

\begin{thm}[Weak skew finitesatz]\label{thm:skewfinsatz} 
Let $J$ be a left ideal in $F[x_1,\ldots,x_n;\sigma]$ with $\mathscr{V}(J) = \emptyset$. Then 
$$F[x_1,\dots,x_n;\sigma]=\mathscr{I}(\mathscr{V}(J))=J+\mathscr{I}(F^{n,\sigma}).$$ 
\end{thm} 
\begin{proof} We must show that $1 \in J+\mathscr{I}(F^{n,\sigma})$. For that, let $A$ be a maximal subset of $F^{n,\sigma}$ for which there exists an element $g \in J$ satisfying $g({\bf a}) = 1$ for all ${\bf a} \in A$.\footnote{since $F^{n,\sigma}$ is finite, such a maximal subset exists.} Let us assume that $g$ is such a polynomial. 

Assume first that $A = F^{n,\sigma}$. In this case $1-g$ vanishes on $F^{n,\sigma}$. Hence $1 = g+(1-g) \in J+\mathscr{I}(F^{n,\sigma})$.

Assume now that $A$ is a proper subset of $F^{n,\sigma}$. Then $g({\bf b}) \neq 1$ for all ${\bf b}  \in F^{n,\sigma}\setminus A$, by the maximality assumption. We now distinguish between two cases.
\vskip 1mm
\noindent
$\bullet$ First suppose that there exists ${\bf b}  \in F^{n,\sigma}\setminus A$ such that ${\bf b}^{1-g({\bf b})} \notin A$. Since $\mathscr{V}(J) = \emptyset$, there exists $h \in J$ such that $h({\bf b}^{1-g({\bf b})}) \neq 0$. By multiplying $h$ from the left by a scalar, we may assume that $h({\bf b}^{1-g({\bf b})})=1$. Then $\big((1-h)(1-g)\big)({\bf a}) = 0$ for all ${\bf a} \in A$. Moreover, by the product formula (Lemma \ref{lem:prodformskewmult}), we have:

$$\big((1-h)(1-g)\big)({\bf b}) = (1-h({\bf b}^{1-g({\bf b})}))(1-g({\bf b})) = (1-1)(1-g({\bf b}))= 0.$$
Consider $\tilde{g} = g+h-hg$. Then $\tilde{g} \in J$ and by the above we have $\tilde{g}({\bf a})=1$ for all ${\bf a} \in A \cup \{{\bf b}\}$, contradicting the maximality of the set $A$. 
\vskip 1mm
\noindent
$\bullet$ Now suppose that ${\bf b}^{1-g({\bf b})} \in A$ for all ${\bf b}  \in F^{n,\sigma}\setminus A$. Then $$(1-g)^2({\bf b}) = (1-g({\bf b}^{1-g({\bf b})})) \cdot (1-g({\bf b})) =(1-1) \cdot (1-g({\bf b})) = 0$$
for all ${\bf b}  \in F^{n,\sigma}\setminus A$. Moreover we have $(1-g)^2({\bf a}) = 0$ for all ${\bf a} \in A$. Thus $(1-g)^2$ vanishes everywhere at $F^{n,\sigma}$, hence $$1 = (2g-g^2)+(1-g)^2\in J+ \mathscr{I}(F^{n,\sigma}),$$
as needed. 

Consequently $$F[x_1,\dots,x_n;\sigma]=\mathscr{I}(\mathscr{V}(J))=J+\mathscr{I}(F^{n,\sigma}).$$ 
\end{proof}

\subsection{On the description of $\mathscr{I}(F^{n,\sigma})$}
In this part, we provide a complete description of the vanishing ideal of the entire $\sigma$-affine space $F^{n,\sigma}$.
The following result is a slight generalization of \cite[Remark 2.4]{Ler12}.
\begin{lem}\label{lem:lcmfinifield}
We have
$$
\lcm(x-a \mid a \in F)=x^{\frac{m}{\theta}(p^\theta-1)+1}-x.
$$
\end{lem}
\begin{proof}
The proof follows the same reasoning as that of \cite[Remark 2.4]{Ler12}.
\end{proof}
\begin{lem}\label{lem:nullvanismalldeg}
Let $h \in \mathscr{I}(F^{n,\sigma})$ be such that:
\begin{enumerate}[i)]
\item
$\deg_{x_{i}}(h) \leq \frac{m}{\theta}(p^\theta-1)$, for all $1 \leq i \leq n$; \label{eq1:lem:nullvanismalldeg}
\item for every $  1 \leq j<i\leq n$, no monomial in $h$ is divisible by $x_ix_j^{p^\theta}$.\label{eq2:lem:nullvanismalldeg}
\end{enumerate}
Then $h=0$.
\end{lem}
\begin{proof}
We prove the result by induction on $n \geq 1$. 

For $n=1$, let $h \in \mathscr{I}(F)$ satisfy \eqref{eq1:lem:nullvanismalldeg} and \eqref{eq2:lem:nullvanismalldeg}. By Lemma \ref{lem:lcmfinifield}, $x_1^{\frac{m}{\theta}(p^\theta-1)+1}-x_1$ divides $h$. Then, from \eqref{eq1:lem:nullvanismalldeg}, we obtain $h=0$. Therefore the theorem holds for $n=1$

Now, assume that the theorem holds for all integers $1,\dots,n-1$ with $n \geq 2$. Let $h \in \mathscr{I}(F^{n,\sigma})$ satisfy \eqref{eq1:lem:nullvanismalldeg} and \eqref{eq2:lem:nullvanismalldeg}. Assume to the contrary that $h \neq 0$. Write 
$$h=x_1^e h_{e}(x_2,\dots,x_n)+x_1^{e-1} h_{e-1}(x_2,\dots,x_n)+\cdots+x_1 h_{1}(x_2,\dots,x_n)+ h_{0}(x_2,\dots,x_n),
$$ where $h_i(x_2,\dots,x_n) \in F[x_2,\dots,x_n;\sigma]$ ($1 \leq i \leq e$) and $h_e(x_2,\dots,x_n)\neq 0$. We consider two cases.
\vskip 1mm
\noindent
$\bullet$ {\bf First suppose that $e\leq p^\theta-1$}. Let $\lambda \in \mathscr{W}_F$. Then, for every ${\bf v}=\omega_\lambda (a_1,\dots,a_n)\in \mathfrak{S}_{\lambda}^n= \omega_\lambda K^n$, using Corollary \ref{cor:evalformon},
\begin{equation}\label{eq:hvv}
\begin{aligned}
0=h({\bf v})
&={\rm N}^\sigma_e(\omega_\lambda a_1)
  \sigma^e\big(h_e(\omega_\lambda(a_2,\dots,a_n))\big)\\
&\quad+\cdots+{\rm N}^\sigma_1(\omega_\lambda a_1)
  \sigma\big(h_1(\omega_\lambda(a_2,\dots,a_n))\big)
  +h_0(\omega_\lambda(a_2,\dots,a_n))\\
&=a_1^e {\rm N}^\sigma_e(\omega_\lambda)
  \sigma^e\big(h_e(\omega_\lambda(a_2,\dots,a_n))\big)\\
&\quad+\cdots+a_1{\rm N}^\sigma_1(\omega_\lambda)
  \sigma\big(h_1(\omega_\lambda(a_2,\dots,a_n))\big)
  +h_0(\omega_\lambda(a_2,\dots,a_n)).
\end{aligned}
\end{equation} Fix $(a_2,\dots,a_n) \in K^{n-1}$. Since $e \leq p^\theta-1$, ${\rm N}^\sigma_e(\omega_\lambda) \neq 0$ and \eqref{eq:hvv} holds for all $a_1 \in K$, we have $h_e(\omega_{\lambda}(a_2,\dots,a_n))=0$. Hence, for every $\lambda \in \mathscr{W}_F$, the polynomial $h_e \in F[x_2,\dots,x_n;\sigma]$ vanishes on $\mathfrak{S}_{\lambda}^{n-1}$. Therefore $h_e$ also vanishes on $F^{n-1,\sigma}=\bigcup_{\lambda \in \mathscr{W}_F} \mathfrak{S}_{\lambda}^{n-1}$. Since $h_e$ also satisfies \eqref{eq1:lem:nullvanismalldeg} and \eqref{eq2:lem:nullvanismalldeg} for all $2 \leq i,j \leq n$, the induction hypothesis implies that $h_e=0$, a contradiction.
\vskip 1mm
\noindent
$\bullet$ {\bf Now suppose that $e>p^\theta-1$}. In this case, we have $c=h_e(x_2,\dots,x_n) \in F^*$; otherwise $h$ would contain $x_jx_1^{p^\theta}$ for some $j>1$, contradicting \eqref{eq2:lem:nullvanismalldeg}. Note that
\begin{align*}
0=h(\omega_\lambda(a_1,0,\dots,0))
&=\sigma^e(c){\rm N}^\sigma_e(\omega_\lambda a_1)\\
&\quad+{\rm N}^\sigma_{e-1}(\omega_\lambda a_1)
  \sigma^{e-1}\big(h_{e-1}({\bf 0})\big)+\cdots\\
&\quad+{\rm N}^\sigma_1(\omega_\lambda a_1)
  \sigma\big(h_1({\bf 0})\big)+h_0({\bf 0}),
\end{align*} for all $\lambda \in \mathscr{W}_{F}$ and all $a_1 \in K$. Letting 
$$g=\sigma^e(c)x_1^e+\sigma^{e-1}\big(h_{e-1}({\bf 0})\big)x_1^{e-1}+\cdots+\sigma\big(h_1({\bf 0})\big)x_1+h_0({\bf 0}) \in F[x_1;\sigma],$$ we see that $g$ vanishes on $F$. By Lemma \ref{lem:lcmfinifield}, $x_1^{\frac{m}{\theta}(p^\theta-1)+1}-x_1$ divides $g$, so $e \geq \frac{m}{\theta}(p^\theta-1)+1$, contradicting \eqref{eq1:lem:nullvanismalldeg}.

In both cases, we reach a contradiction. Consequently, we must have $h=0$, completing the proof by induction.
\end{proof}
\begin{thm}[Vanishing ideal of the $\sigma$-affine space]\label{thm:descrvanide}
We have
\begin{equation}\label{eq:desideasigmfqn}
\mathscr{I}(F^{n,\sigma})=\left\langle x_ix_j^{p^\theta}-x_i^{p^\theta}x_j, x_i^{\frac{m}{\theta}(p^\theta-1)+1}-x_i \bigg\vert 1 \leq i,j \leq n \right\rangle.
\end{equation}
\end{thm}
\begin{proof}
Denote by $J$ the left ideal on the RHS of \eqref{eq:desideasigmfqn}. 

First, we show that $J \subset \mathscr{I}(F^{n,\sigma})$. By Lemma \ref{lem:lcmfinifield}, the polynomial $x_i^{\frac{m}{\theta}(p^\theta-1)+1}-x_i$ vanishes on $F^{n,\sigma}$, for all $1 \leq i \leq n$. Additionally, for $1 \leq i,j \leq n$, a straightforward computation shows that the polynomial $x_ix_j^{p^\theta}-x_i^{p^\theta}x_j$ vanishes on each $\mathfrak{S}_{\lambda}^{n}=\omega_\lambda K^n$ ($\lambda  \in \mathscr{W}_{F}$), and so it also vanishes on $F^{n,\sigma}=\bigcup_{\lambda \in \mathscr{W}_F}\mathfrak{S}_{\lambda}^n$. Therefore $J \subset \mathscr{I}(F^{n,\sigma})$.

Conversely let $f \in \mathscr{I}(F^{n,\sigma})$. We claim that there exists $g \in J$ such that: 
\begin{enumerate}[i)]
\item
$\deg_{x_{i}}(f-g) \leq \frac{m}{\theta}(p^\theta-1)$, for all $1 \leq i \leq n$;
\item
and for every $1 \leq j < i \leq n$, no monomial in $f-g$ is divisible by $x_ix_j^{p^\theta}$;
\end{enumerate}
Indeed, we can repeatedly replace (modulo $J$) each $x_i^v$ ($v \geq 0$) in $f$ by the remainder of its right-hand division by $x_i^{\frac{m}{\theta}(p^\theta-1)+1}-x_i$; and also replace (modulo $J$) each $x_ix_j^{p^\theta}$ ($j < i$) in $f$ by $x_i^{p^\theta}x_j$. Since $f-g$ also vanishes on $F^{n,\sigma}$, Lemma \ref{lem:nullvanismalldeg} implies that $f-g=0$, and thus $f=g \in J$. Therefore $\mathscr{I}(F^{n,\sigma}) \subset J$.

Consequently, we obtain 
$$\mathscr{I}(F^{n,\sigma})=J=\left\langle x_ix_j^{p^\theta}-x_i^{p^\theta}x_j, x_i^{\frac{m}{\theta}(p^\theta-1)+1}-x_i \bigg\vert 1 \leq i,j \leq n \right\rangle,
$$ as was to be proved.
\end{proof}
\subsection{Skew Finitesatz: The one-variable case}
In this part, we prove Skew Finitesatz for one variable polynomials. Let $x$ be a variable.
\begin{thm}[One-variable skew Finitesatz]\label{thm:skewonedim}
Let $J$ be a nonzero left ideal of $F[x;\sigma]$. Then 
$$\mathscr{I}(\mathscr{V}(J))= J+F[x;\sigma]\cdot \left(x^{\frac{m}{\theta}(p^\theta-1)+1}-x\right)
.$$
\end{thm}
\begin{proof}
Since $F[x;\sigma]$ is left principal, there exists $f \in F[x;\sigma]$ such that $J= F[x;\sigma] \cdot f$. Set $p_J=\lcm(x-a \mid a \in \mathscr{V}(J))$. Note that $g$ vanishes on $\mathscr{V}(J)$ if and only if $g \in F[x;\sigma] \cdot p_J$. Since $F[x;\sigma]$ is left principal, there exists a unique monic polynomial $h \in F[x;\sigma] \setminus \{0\}$ such that
$$
F[x;\sigma]\cdot f + F[x;\sigma] \cdot \left(x^{\frac{m}{\theta}(p^\theta-1)+1}-x\right)=F[x;\sigma]\cdot h. 
$$It remains to prove that $h=p_J$. Clearly $h$ vanishes on $\mathscr{V}(J)$, so $p_J$ divides $h$. As $h$ divides $x^{\frac{m}{\theta}(p^\theta-1)+1}-x$, by \cite[Theorem 5.1]{LL04} , there exists a subset $A \subset F$ such that $h=\lcm(x-a \mid a \in A)$. Hence $f$ vanishes on $A$, and so $A \subset \mathscr{V}(J)$. Therefore $h$ divides $p_J$. Consequently $h=p_J$. 
\end{proof}
\section{Open questions}
In this section, we collect open questions for further investigation. Regarding the skew Chevalley--Warning theorem, we pose the following.
\begin{prob}
\begin{itemize}
\item Is the bound $n\frac{q^{\frac{1}{o(\sigma)}}-1}{q-1}$ in Theorem \ref{thm:skewchevwar} optimal when $\sigma$ is a nontrivial automorphism? As in the classical case, it is interesting to ask whether Theorem \ref{thm:skewchevwar} can also be deduced from the skew combinatorial Nullstellensatz (Theorem \ref{new2}).
\item In \cite{LP23}, Leep and Petrik present refinements of the classical Chevalley--Warning theorem concerning the lower bound on $\mathscr{V}(f_1,\dots,f_r)$.  Is it possible to obtain a similar improvement in the skew setting, and thereby strengthen Corollary \ref{cor:effnull}?
\end{itemize}
\end{prob}
Regarding the skew Finitesatz, we pose the following.
\begin{prob}
\begin{itemize}
\item Can we obtain a strong skew Finitesatz in the multivariate case? Using the same notation as in section \ref{sec:skewfinitesatz}, is it true that, for any left ideal $J$ in $F[x_1,\dots,x_n;\sigma]$, we have
$$
\mathscr{I}(\mathscr{V}(J))=J+\mathscr{I}(F^{n,\sigma})?
$$
\item In \cite[Theorem 7]{Cla14}, Clark proved a Finitesatz over an arbitrary field-- namely a result about the zeros of an ideal on a finite subset of $F^n$. Can this be extended to the skew setting? 
\end{itemize}
\end{prob}
\bibliography{biblio_ver2}
\bibliographystyle{amsalpha}
\end{document}